\documentclass{daj}

\usepackage{dsfont} 
\usepackage{comment}

\usepackage{hyperref}

\usepackage{amsthm, amsmath}
\usepackage{amssymb}
\usepackage{enumitem}
\usepackage[capitalize,nosort]{cleveref}
\usepackage{color}

\usepackage{mathtools}
\usepackage{faktor} 

\crefname{enumi}{}{} 

\newcounter{maincounter}[section] 

\newtheorem{theorem}[maincounter]{Theorem}
\newtheorem{lemma}[maincounter]{Lemma}
\newtheorem{proposition}[maincounter]{Proposition}
\newtheorem{corollary}[maincounter]{Corollary}

\theoremstyle{remark}
\newtheorem{remark}[maincounter]{Remark}

\theoremstyle{definition}
\newtheorem{definition}[maincounter]{Definition}

\newcommand{\N}{\ensuremath{\mathbb N}} 
\newcommand{\Z}{\ensuremath{\mathbb Z}} 
\newcommand{\Q}{\ensuremath{\mathbb Q}} 
\newcommand{\R}{\ensuremath{\mathbb R}} 
\providecommand{\C}{}
\renewcommand{\C}{\ensuremath{\mathbb C}} 
\newcommand{\F}{\ensuremath{\mathbb F}} 

\newcommand{\defeq}{\ensuremath{\coloneqq}}

\newcommand\rk{\ensuremath{\operatorname{rk}}} 
\newcommand{\FS}{\ensuremath{\operatorname{FS}}} 
\newcommand{\im}{\ensuremath{\operatorname{Im}\,}} 
\newcommand{\Hom}{\ensuremath{\operatorname{Hom}}} 
\newcommand{\cyclic}[1]{\ensuremath{\Z / #1\Z}}		
\newcommand{\cyclicstar}[1]{\ensuremath{(\Z / #1\Z)^\times}} 
\newcommand{\emptyparam}{\ensuremath{\,\cdot\,}}

\newcommand{\restricts}[2] {
	#1 
	\raisebox{-.3ex}{$|$}_{#2}
}

\DeclareMathOperator{\Span}{span}
\DeclareMathOperator{\supp}{supp}

\dajAUTHORdetails{%
  title = {Reconstructing a Set from its Subset Sums: $2$-Torsion-Free Groups}, 
  author = {Federico Glaudo and Noah Kravitz},
  plaintextauthor = {Federico Glaudo and Noah Kravitz},
  plaintexttitle = {Reconstructing a Set from its Subset Sums: 2-Torsion-Free Groups}, 
  keywords = {subset sums, reconstruction, inverse problems in additive combinatorics},
}   

\dajEDITORdetails{%
   year={2024},
   number={14},
   received={26 May 2023},   
   published={10 December 2024},  
   doi={10.19086/da.125856},       
}   

\begin{document}

\begin{frontmatter}[classification=text]


\author[fede]{Federico Glaudo\thanks{supported by the National Science Foundation under Grant No. DMS--1926686}}
\author[noah]{Noah Kravitz\thanks{supported in part by an NSF Graduate
Research Fellowship (grant DGE--2039656)}}

\begin{abstract}
For a finite multiset $A$ of an abelian group $G$, let $\text{FS}(A)$ denote the multiset of the $2^{|A|}$ subset sums of $A$.  
It is natural to ask to what extent $A$ can be reconstructed from $\text{FS}(A)$.  We fully solve this problem for $2$-torsion-free groups $G$ by giving characterizations, both algebraic and combinatorial, of the fibers of $\text{FS}$.
Equivalently, we characterize all pairs of multisets $A,B$ with $\text{FS}(A)=\text{FS}(B)$.
Our results build on recent work of Ciprietti and the first author.
\end{abstract}
\end{frontmatter}

\section{Introduction}

\subsection{The reconstruction problem}
Let $G$ be an abelian group.  Given a finite multiset\footnote{See Section 2.1 of \cite{CipriettiGlaudo2023} for a review of notation and terminology for multisets.} $A$ of $G$, we write $\sum A$ for the sum of the elements of $A$, counted with multiplicity.  (In this paper, all multisets are finite.)  If $A$ is a multiset of $G$, then its \emph{multiset of subset sums}
$$\FS(A)\defeq \big\{\sum A': A' \subseteq A\big\}$$
is the multiset consisting of the $2^{|A|}$ subset sums of $A$.  It is natural to ask whether $A$ can be uniquely reconstructed from $\FS(A)$.  The answer in general is ``no'': For instance, if $\sum A=0$, then $\FS(A)=\FS(-A)$.  More generally, if $\sum A=0$, then for any $B$ we have $\FS(A \cup B)=\FS((-A) \cup B)$.

Ciprietti and the first author \cite{CipriettiGlaudo2023} recently showed that for a large class of abelian groups $G$, the multiset $A$ can be reconstructed from $\FS(A)$ ``up to'' these symmetries of flipping the signs of a zero-sum subset.  Let $O_{\FS}$ denote the set of odd natural numbers $n$ such that $(\cyclic{n})^\times$ is (multiplicatively) generated by the elements $2, -1$.

\begin{theorem}[{\cite[Theorem 1.1]{CipriettiGlaudo2023}}]\label{thm:old}
Let $G$ be an abelian group such that every element with finite order has order in $O_{\FS}$, and let $A,B$ be finite multisets of $G$.  Then $\FS(A)=\FS(B)$ if and only if $B$ can be obtained from $A$ by flipping the signs of a zero-sum subset of $A$.
\end{theorem}

In the same paper, they showed that if $G$ has any element with finite order not lying in $O_{\FS}$, then there exist multisets $A,B$ of $G$ such that $\FS(A)=\FS(B)$ but $B$ cannot be be obtained from $A$ by flipping the signs of a zero-sum subset.  This observation leads to a more general reconstruction problem: For an arbitrary abelian group $G$, up to which symmetries can $A$ be recovered from $\FS(A)$?  The purpose of the present paper is to provide a complete answer to this question for groups with no $2$-torsion.

\subsection{Main results}
There are several easily-described ``moves'' whose application to the multiset $A$ preserves $\FS(A)$ up to replacing it by some shift $\FS(A)+s$.  Our main result will show that, when the group $G$ has no $2$-torsion, these moves in fact generate all of the multisets $B$ satisfying $\FS(A)=\FS(B)+s$ for some $s \in G$. 
 It will then be easy to determine which such sequences of moves lead to multisets $B$ with $s=0$.

\begin{definition}\label{def:u}
    For $n\ge 3$ an odd natural number, choose the smallest $k \ge 1$ such that $2^k$ is equivalent to $-1$ or $1$ modulo $n$, and define the set $U_n\defeq \{1,2,4,\ldots, 2^{k-1}\} \subseteq \cyclic{n}$.
\end{definition}

\begin{definition}\label{def:moves}
    Let $G$ be an abelian group with no elements of order $2$, and let $A$ be a finite multiset of $G$.  We define the following two types of \emph{moves}:
    \begin{itemize}
        \item $g\leadsto -g$: For an element $g\in A$, replace $g$ with $-g$; the set $A$ becomes $A\cup\{-g\}\setminus\{g\}$.
        \item $\iota(\alpha U_n)\leadsto \iota(\beta U_n)$: For $n\ge 3$ odd, $\alpha,\beta\in\cyclicstar{n}$, and an embedding $\iota:\cyclic{n}\to G$ such that $\iota(\alpha U_n)\subseteq A$, replace the subset $\iota(\alpha U_n)$ of $A$ with $\iota(\beta U_n)$; the set $A$ becomes $A\cup \iota(\beta U_n)\setminus\iota(\alpha U_n)$.
    \end{itemize}
\end{definition}

A quick computation (see \cref{lem:fs-of-moves}) shows that these moves preserve $\FS(A)$ up to a shift. 

Before we can state our main results, we need one further definition.

\begin{definition}\label{def:v}
    Given an abelian group $G$ with no elements of order $2$, define
    \begin{equation*}
        V(G) \defeq  
        \left\{\mu:G\to\Z :\, 
        \begin{aligned}
           &\mu(g)+\mu(-g)=\mu(2g) + \mu(-2g) \text{ for all $g\in G$},
           \\
           &\sum_{g\in H} \mu(g) = 0 \text{ for all subgroups $H\le G$},
           \\
           &\supp(\mu) \text{ is finite}
        \end{aligned}
        \right\}.
    \end{equation*}
\end{definition}

Our main theorem is the following three-way equivalence.  The \emph{multiplicity function} $\mu_S: G \to \N$ of a finite multiset $S$ of $G$  is the function whose value at $g$ is the multiplicity of $g$ in $S$.

\begin{theorem}\label{thm:main}
Let $G$ be an abelian group with no elements of order $2$, and let $A,B$ be finite multisets of $G$.  Write $\mu\defeq \mu_A-\mu_B$.  Then the following are equivalent:
\begin{enumerate}[label=(\roman*)]
    \item $\FS(A)=\FS(B)+s$ for some $s \in G$.
    \item $\mu\in V(G)$.
    \item There is a sequence of multisets $A=E_0, E_1, E_2, \dots, E_m=B$ such that each
    $E_i$ can be transformed into $E_{i+1}$ by applying one of the moves described in \cref{def:moves}. Moreover, for each $i$ we have $\FS(E_{i+1})=\FS(E_i)+s_i$ for some $s_i\in G$.
\end{enumerate}
\end{theorem}

With the characterization from \cref{thm:main} in hand, we can quickly specialize to pairs of multisets whose multisets of subset sums are equal (with no shift).

\begin{theorem}\label{thm:main-no-shift}
Let $G$ be an abelian group with no elements of order $2$, and let $A,B$ be finite multisets of $G$.  Write $\mu\defeq \mu_A-\mu_B$.  Then the following are equivalent:
\begin{enumerate}[label=(\roman*)]
    \item $\FS(A)=\FS(B)$.
    \item $\mu\in V(G)$ and $\sum_{g\in G} \mu(g)g=0$.
    \item There is a sequence of multisets $A=E_0, E_1, E_2, \dots, E_m=B$ such that each
    $E_i$ can be transformed into $E_{i+1}$ by applying one of the moves described in \cref{def:moves}.  Moreover, for each $i$ we have $\FS(E_{i+1})=\FS(E_i)+s_i$ for some $s_i\in G$, and $s_0+s_1+\cdots + s_{m-1}=0$.
\end{enumerate}
\end{theorem}

In our proof of \cref{thm:main}, we will establish the three implications (i)$\implies$(ii), (ii)$\implies$(iii), and (iii)$\implies$(i).  For the implication (i)$\implies$(ii), which is the most difficult, our strategy is a refinement of the approach in \cite{CipriettiGlaudo2023}.  We will start with the case where $G$ is a finite cyclic group; here, after a Fourier-analytic reduction, we will use arguments about the ranks of groups of cyclotomic units.  We will then use the discrete Radon transform to extend to the case where $G$ is of the form $(\cyclic{n})^r$, and finally we will treat the case of general $G$.  We emphasize that two important technical ingredients in the implication (i)$\implies$(ii), namely, \cref{lem:fshat-surjective} (that the image of $\widehat \FS$ has full rank) and \cref{thm:radon-inversion} (the inversion formula for the Radon transform), previously appeared in \cite{CipriettiGlaudo2023}; we are able to use these tools as black boxes.  The implication (ii)$\implies$(iii), which was trivial in the setting considered in \cite{CipriettiGlaudo2023}, follows from an in-depth study of the $\Z$-module $V(G)$.
For both of these implications, we will make use of several equivalent descriptions of $V(G)$;
this attention to the properties of the ``kernel'' of $\FS$ is the site of many of the new ideas of this paper.
The implication (iii)$\implies$(i) is immediate.

\cref{thm:main,thm:main-no-shift} completely resolve our main question for groups $G$ with no elements of order $2$.  We briefly mention how different behavior can appear when $G$ has $2$-torsion.  The simplest example of what can go ``wrong'' is visible for the group $G=(\cyclic{2})^r$: Here, one can show that $\FS(A)=\FS(B)$ if and only if $|A|=|B|$ and $\Span(A)=\Span(B)$.\footnote{The concept of span is induced by the $\F_2$-linear structure of $(\cyclic{2})^r$.}  This condition appears not to be expressible in terms of properties of $\mu_A-\mu_B$ alone---so not only do the proof techniques of the present paper break down in the presence of $2$-torsion, but one should expect the answer to the main question to have a different ``shape''.

\subsection{Related work}
Let us explain how \cref{thm:main-no-shift} implies the main result of \cite{CipriettiGlaudo2023}.  Suppose $G$ is an abelian group where every element of finite order has order in the set $O_{\FS}$.  This condition is equivalent to the assertion that all embedded $\cyclic{n}$'s in $G$ have $n \in O_{\FS}$.  In particular, for such an $n \neq 1$, we have $U_n \sqcup -U_n=\cyclicstar{n}$, and so moves of the second type in \cref{def:moves} can be obtained using only moves of the first type.  Thus, (iii) of \cref{thm:main-no-shift} tells us that if $\FS(A)=\FS(B)$, then $B$ and be obtained from $A$ by flipping the signs of some subset, and the condition $\sum_{g\in G}\mu(g)g=0$ (i.e., $\sum A=\sum B$) from (ii) ensures that the flipped subset has sum zero.

In a recent paper, Gaitanas \cite{Gaitanas2023} studied the problem of classifying the multisets $A$ of $\cyclicstar{n}$ such that $\FS(A)\setminus\{0\}$ is uniformly distributed on $\cyclic{n}$.  He showed that, for some particular odd values of $n$, such a set $A$ must be a union of sets of the form $\alpha\{\pm 1, \pm 2, \pm 4, \dots, \pm 2^{r-1}\}$, where $r$ is either the multiplicative order of $2$ modulo $n$ or half of the order,\footnote{The statement of \cite[Theorem 2.3]{Gaitanas2023} claims that one can choose $r$ equal to the order of $2$, but this is a mistake. The proof works if one allows $r$ to be half of the order.} and the signs $\pm$ can be chosen independently.  Our \cref{thm:main-no-shift} implies this result for any odd $n\ge 1$. Indeed, if $\FS(A)\setminus\{0\}$ is equidistributed, then $n$ divides $ 2^{|A|}-1$ and thus the multiplicative subgroup generated by $2$ has order dividing $|A|$.  Let $B$ be the multiset consisting of $|A|/|\langle 2 \rangle_{\cyclicstar{n}}|$ copies of $\langle 2 \rangle_{\cyclicstar{n}}$.  An easy computation shows that $\FS(A)=\FS(B)$, and then condition (iii) of \cref{thm:main-no-shift} implies the desired constraints on $A$.  We also remark that Gaitanas pointed out an interesting connection with questions about the linear dependences among logarithms of sines of rational multiples of $\pi$.

For $t \geq 2$ a natural number, the \emph{reconstruction problem for $t$-subset sums} asks if one can reconstruct $A$ from the multiset of the $\binom{|A|}{t}$ sums of its size-$t$ subsets.  This problem has been studied extensively \cite{SelfridgeStraus1958,GordonFraenkelStraus1962,BomanLinusson1996,AmdeberhanZeleke1997}, and for a more thorough literature review the reader should consult the introduction of \cite{CipriettiGlaudo2023} or the surveys \cite{fomin2019,Nathanson2008}.

\subsection{Organization of the paper}
In \cref{sec:V(g)}, we gather several facts about the $\Z$-modules $V(G)$; phrasing these facts in more abstract language will be very convenient for applications in the remainder of the paper.  In \cref{sec:finite-cyclic}, we prove the implication (i)$\implies$(ii) of \cref{thm:main} for finite cyclic groups.  In \cref{sec:radon}, we prove this implication for all abelian groups.  Finally, in \cref{sec:final}, we put everything together and deduce \cref{thm:main,thm:main-no-shift}.

\section{\texorpdfstring{Properties of the $\Z$-modules $V(G)$}{Properties of the Z-module V(G)}}\label{sec:V(g)}
In this section we will collect several properties and characterizations of the $\Z$-modules $V(G)$ (see \cref{def:v}).  Throughout, we use the convention $\cyclicstar{1}=\cyclic{1}$.

\subsection{Decomposition lemmas}

The following ``decomposition lemmas'' will be useful for analyzing $V(G)$. 

\begin{definition}
For $n\ge 1$ an odd natural number, let
    \begin{equation*}
        \tilde V(\cyclic{n}) \defeq 
        \left\{\mu:\cyclicstar{n} \to\Z :\, 
        \begin{aligned}
           &\mu(g)+\mu(-g)=\mu(2g) + \mu(-2g) \text{ for all $g\in \cyclicstar{n}$},
           \\
           &\sum_{g\in \cyclicstar{n}} \mu(g) = 0
        \end{aligned}
        \right\}.
    \end{equation*}
\end{definition}

An abelian group is called a \emph{torsion group} if all of its elements have finite order.

\begin{lemma}\label{lem:v-decomposition}
    Let $G$ be an abelian \emph{torsion} group with no elements of order $2$.  Then the map $\mu\mapsto (\restricts{\mu}{H^{\times}})_{H\le G \text{ cyclic}}$ defines a ($\Z$-module) isomorphism
    \begin{equation*}
        V(G) \xrightarrow{\sim} \bigoplus_{H\le G \text{ cyclic}} \tilde V(H).
    \end{equation*}
\end{lemma}
\begin{proof}
Each element $g \in G$ generates a finite cyclic group $H_g$ which is of the form $\iota (\cyclic{n})$, where $n$ is the (odd) order of $g$ and $\iota$ is an embedding of $\cyclic{n}$ into $G$.  Clearly $g \in H^\times_g$.  
For any $g,g' \in G$, the sets $H^\times_g, H^\times_{g'}$ are either equal or disjoint, according to whether or not $H_g=H_{g'}$.  Hence we can write
    \begin{equation}\label{eq:g-decomposition}
        G=\bigsqcup_{H\le G \text{ cyclic}} H^\times,
    \end{equation}
and this partition induces an isomorphism $$\Phi:\Z^G\xrightarrow{\sim}\bigoplus_{H\le G \text{ cyclic}} \Z^{H^\times}$$ via
    \begin{equation*}
        \Phi(\mu) = (\restricts{\mu}{H^\times})_{H\le G \text{ cyclic}} .
    \end{equation*}

    We claim that $\Phi$ restricts to an isomorphism from $V(G)$ to the direct sum
    $\oplus_{H\le G \text{ cyclic}} \tilde V(H)$.
    Let us show that if $\mu\in V(G)$, then $\restricts{\mu}{H^\times}\in \tilde V(H)$ for all cyclic subgroups $H\le G$.
    First, for any $g\in H^\times$, one has $\restricts{\mu}{H^\times}(g) + \restricts{\mu}{H^\times}(-g) = \mu(g)+\mu(-g) = \mu(2g)+\mu(-2g) = \restricts{\mu}{H^\times}(2g) + \restricts{\mu}{H^\times}(-2g)$.  
    Next, to show that $\sum_{g \in H^\times}\restricts{\mu}{H^\times}(g)=0$, we apply the Inclusion-Exclusion Principle to the subgroups of $H$ as follows. Recall that $H=\iota(\cyclic{n})$ for some odd $n\ge 1$ and some embedding $\iota$ of $\cyclic{n}$ into $G$. We have
    \begin{equation*}
   \mathds{1}_{\cyclicstar{n}} = \sum_{d\mid n} \mu_{\text{m\"obius}}(d){\mathds 1}_{d\cyclic{n}},
    \end{equation*}
    where $\mu_{\text{m\"obius}}$ denotes the M\"obius function, and so
    \begin{equation*}
        \sum_{g\in H^\times} \restricts{\mu}{H^\times}(g) =
        \sum_{g \in \iota(\cyclic{n})} \mathds{1}_{H^\times}(g)\mu(g) = 
        \sum_{d\mid n}
        \mu_{\text{m\"obius}}(d)
        \sum_{g\in \iota(d\cyclic{n})} \mu(g) = 0.
    \end{equation*}
Hence $\Phi$ restricts to a $\Z$-module isomorphism from $V(G)$ onto its image, which is contained in $\oplus_{H\le G \text{ cyclic}} \tilde V(H)$.  It remains to show that this image is in fact all of $\oplus_{H\le G \text{ cyclic}} \tilde V(H)$.  

Let $(\mu^H)_{H\le G \text{ cyclic}} \in \oplus_{H\le G \text{ cyclic}} \tilde V(H)$.  Then there is a unique $\mu \in \Z^G$ such that $\restricts{\mu}{H^\times}=\mu^H$ for all cyclic $H \leq G$.  The support of $\mu$ is automatically finite due to the definition of the direct sum.  The proof will be complete if we show that $\mu \in V(G)$.  Now the only nontrivial fact to check is that $\sum_{g\in H}\mu(g) = 0$ for all subgroups $H$. This follows from the identity, analogous to \cref{eq:g-decomposition}, 
    \begin{equation*}
        \sum_{g\in H} \mu(g) = \sum_{H'\le H\text{ cyclic }} \sum_{g'\in (H')^\times}\mu(g').
    \end{equation*}
\end{proof}

To complement this lemma on torsion groups, our next lemma shows that for a general abelian group $G$ all of the ``action'' of $V(G)$ occurs in the torsion subgroup $G_T$ (the subgroup consisting of all elements with finite order).

\begin{lemma}\label{lem:v-nontorsion-decomposition}
    Let $G$ be an abelian group with no elements of order $2$ and let $G_T$ be its torsion subgroup. 
    Then the map $\mu\mapsto (\restricts{\mu}{G_T}, \restricts{\mu}{G\setminus G_T})$ defines a ($\Z$-module) isomorphism
    \begin{equation*}
        V(G) \xrightarrow{\sim} V(G_T)\oplus
        \left\{\mu:G\setminus G_T \to\Z :\, 
        \begin{aligned}
           &\mu(g)+\mu(-g)=0 \text{ for all $g\in G\setminus G_T$},
           \\
           & \supp(\mu) \text{ is finite}
        \end{aligned}
        \right\}.
    \end{equation*}
\end{lemma}
\begin{proof}
    Fix $\mu\in V(G)$. For any $g\in G$ and $i \in \N$, we have $\mu(g)+\mu(-g)=\mu(2^ig)+\mu(-2^ig)$.  If $g$ has infinite order, then we must have $\mu(g)+\mu(-g)=0$ since $\mu$ has finite support. In particular, $\restricts{\mu}{G\setminus G_T}$ belongs to the last $\Z$-module appearing in the statement of the lemma.  It follows directly from the definitions of $V(G)$ and $V(G_T)$ that $\restricts{\mu}{G_T}\in V(G_T)$.   Hence the map defined in the statement of the lemma is a homomorphism.  This map is clearly injective, and it remains to show that it is surjective.

    Surjectivity amounts to the statement that if $\mu:G\to\Z$ with finite support satisfies $\restricts{\mu}{G_T}\in V(G_T)$ and $\mu(g)+\mu(-g)=0$ whenever $g$ has infinite order, then $\mu\in V(G)$. The condition $\mu(g)+\mu(-g)=\mu(2g)+\mu(-2g)$ is readily verified by considering separately the cases $g\in G_T$ and $g\in G\setminus G_T$.  And if $H \leq G$ is a subgroup, then
    $$\sum_{g \in H} \mu(g)=\sum_{g \in H \cap G_T}\mu(g)+\sum_{
    g \in H \setminus G_T} \mu(g)=0,$$
where the first sum vanishes by the assumption $\restricts{\mu}{G_T}\in V(G_T)$ (note that $H \cap G_T$ is a subgroup of $G_T$) and the second sum vanishes because the terms $\mu(g), \mu(-g)$ cancel in pairs.  Hence $\mu \in V(G)$.
\end{proof}

Our next task is computing the rank\footnote{Let us recall that the rank of a $\Z$-module (or of an abelian group) is the maximal size of a set of $\Z$-independent elements (see \cite[Chapter I]{lang2002} for a reference on properties of the rank).} of $V(\cyclic{n})$.  In light of \cref{lem:v-decomposition}, it suffices to compute the rank of each $\tilde V(\cyclic{n})$ individually.  Recall the set $U_n$ from \cref{def:u} for $n \geq 3$ odd.  It will also be useful to define the set $U_n^\pm \defeq  U_n \sqcup (-U_n)$ (which is indeed a disjoint union).

\begin{lemma}\label{lem:dimension-tildev}
    We have $\rk(\tilde V(\cyclic{1}))=0$.  For $n\ge 3$ odd, we have
    \begin{equation*}
        \rk\big(\tilde V(\cyclic{n})) = \frac{\varphi(n)}2 + \frac{\varphi(n)}{2|U_n|} - 1,
    \end{equation*}
    where $\varphi$ denotes Euler's totient function.
\end{lemma}

\begin{proof}
The first statement is trivial, so we will consider only $n\ge 3$.  Define the $\Z$-module homomorphism $F:\tilde V(\cyclic{n})\to\Z^{\cyclicstar{n}}$ via
    \begin{equation*}
        F(\mu)(g) \defeq \mu(g) + \mu(-g).
    \end{equation*}
The definition of $\tilde V(\cyclic{n})$ implies that a function is in the image of $F$ if and only if it assumes a constant value on each coset of $U_n^\pm$ in $\cyclicstar{n}$ and these constants sum to $0$.
Hence $\im F$ has rank $\frac{\varphi(n)}{2|U_n|}-1$ (because there are $\frac{\varphi(n)}{2|U_n|}$ cosets of $U_n^\pm$).  Next, the kernel of $F$ consists of the functions $\mu:\cyclicstar{n}\to\Z$ satisfying $\mu(g)+\mu(-g)=0$ for all $g\in\cyclicstar{n}$.  Thus $\rk(\ker(F))=\frac{\varphi(n)}2$, and putting everything together gives
    \begin{equation*}
        \rk(\tilde V(\cyclic{n}))
        =
        \rk(\ker F) + \rk(\im F) 
        = \frac{\varphi(n)}2 + \frac{\varphi(n)}{2|U_n|}-1.
    \end{equation*}
\end{proof}

\begin{corollary}\label{cor:dimension-v}
    We have $\rk(V(\cyclic{1})) = 0$.  For  $n\ge 3$ odd, we have
    \begin{equation*}
        \rk\big(V(\cyclic{n})\big) = 
        \frac{n-1}2 + \sum_{\substack{d\mid n \\ d\not= 1}}\Big(\frac{\varphi(d)}{2|U_d|}-1\Big),
    \end{equation*}
    where $\varphi$ denotes Euler's totient function.
\end{corollary}
\begin{proof}
    The cyclic subgroups of $\cyclic{n}$ are $(n/d)\Z/n\Z \cong \cyclic{d}$ for the divisors $d$ of $n$.
    Thus, by \cref{lem:v-decomposition,lem:dimension-tildev}, we have
    \begin{align*}
        \rk\big(V(\cyclic{n})\big)
        &=
        \sum_{d\mid n} \rk\big(\tilde V(\cyclic{d}\big)
        = 
        \sum_{1<d\mid n} \left(\frac{\varphi(d)}2 + \frac{\varphi(d)}{2|U_d|} - 1 \right)
        \\
        &=
        \frac{n-1}2 + \sum_{1<d\mid n} \left(\frac{\varphi(d)}{2|U_d|} - 1 \right).
    \end{align*}
\end{proof}

\subsection{Functoriality}
The $\Z$-modules $V(G)$ satisfy two nice ``functoriality'' properties that will later allow us to verify membership in $V(G)$ without repeating tedious computations.

\begin{definition}\label{def:pushforward-pullback}
    Let $G_1, G_2$ be abelian groups, and let $\psi:G_1\to G_2$ be a homomorphism.
    \begin{enumerate}
        \item For any function $\mu:G_1\to\Z$ with finite support, the \emph{pushforward} of $\mu$ by $\psi$ is the function $\psi_*\mu:G_2\to\Z$ given by
        \begin{equation*}
            \psi_*\mu(g_2) \defeq \sum_{g_1\in \psi^{-1}(g_2)} \mu(g_1).
        \end{equation*}
        \item For any function $\mu:G_2\to\Z$, the \emph{pullback} of $\mu$ by $\psi$ is the function $\psi^*\mu:G_1\to\Z$ given by $\psi^*\mu \defeq \mu\circ\psi$.
    \end{enumerate}
\end{definition}

\begin{lemma}\label{lem:v-functoriality}
    Let $G_1, G_2$ be abelian groups with no elements of order $2$, and let $\psi:G_1\to G_2$ be a homomorphism.
    \begin{enumerate}
        \item If $\mu\in V(G_1)$, then $\psi_*\mu\in V(G_2)$.
        \item If $\psi$ has finite kernel and $\mu\in V(G_2)$, then $\psi^*\mu\in V(G_1)$.
    \end{enumerate}
\end{lemma}
\begin{proof}
    Let us prove the statements for the pushforward and pullback separately.
    \begin{enumerate}
        \item It is clear that $\psi_* \mu$ has finite support.  If $H\le G_2$ is a subgroup, then its preimage $\psi^{-1}(H)$ is a subgroup of $G_1$, so $\sum_{g_2 \in H} \psi_*\mu(g_2)=0$. Hence, in order to verify that $\psi_*\mu \in V(G_2)$ it remains only to check that 
        \begin{equation}\label{eq:tmp432}
            \psi_*\mu(g_2)+\psi_*\mu(-g_2)=\psi_*\mu(2g_2) + \psi_*\mu(-2g_2)
        \end{equation} 
        for all $g_2\in G_2$.  Let $(G_1)_T$ denote the torsion subgroup of $G_1$, and consider the decomposition $\mu=\mu|_{(G_1)_T}+\mu|_{G_1 \setminus (G_1)_T}$, as in \cref{lem:v-nontorsion-decomposition}.  The pushforward of $\mu|_{G_1 \setminus (G_1)_T}$ clearly satisfies \eqref{eq:tmp432} since $\mu(g_1)+\mu(-g_1)=0$ for all $g_1 \in G_1 \setminus (G_1)_T$.
        
        To deal with the pushforward of $\mu|_{(G_1)_T}$, consider the restriction of $\psi$ to $(G_1)_T$ and notice that $(\psi|_{(G_1)_T})^{-1}(2g_2) = 2(\psi|_{(G_1)_T})^{-1}(g_2)$. 
        This holds because $g\mapsto 2g$ is an automorphism of $(G_1)_T$ and of $(G_2)_T$ since they are torsion abelian groups without $2$-torsion.
        Now \cref{eq:tmp432} follows from the definition of the pushforward because summands cancel in quadruples
        \begin{align*}
            \psi_*\mu(g_2)+\psi_*\mu(-g_2) &= 
            \sum_{g_1 \in \psi^{-1}(g_2)} \mu(g_1) + \mu(-g_1)
            \\
            &=
            \sum_{g_1 \in \psi^{-1}(g_2)} \mu(2g_1) + \mu(-2g_1)
            \\
            &=
            \sum_{g_1 \in \psi^{-1}(2g_2)} \mu(g_1) + \mu(-g_1)
            =
            \psi_*\mu(2g_2) + \psi_*(-2g_2).
        \end{align*}
        \item The assumption of $\ker\psi$ being finite guarantees that $\psi^*\mu$ has finite support. The other two conditions of membership in $V(G_1)$ follow immediately from the definition of the pullback.
    \end{enumerate}
\end{proof}

\subsection{Relation to moves}
The attentive reader may wonder why the second type of move in \cref{def:moves} uses the set $U_n$ rather than the entire multiplicative subgroup generated by the element $2$.  (These sets differ exactly when $-1$ is a power of $2$ modulo $n$, in which case $\langle 2 \rangle_{\cyclicstar{n}}=U_n \sqcup -U_n$.)  The reason will become clear in the proof of the following proposition, which shows that every element of $V$ can be decomposed as a sum of (indicator functions of) several moves from \cref{def:moves}.  Indeed, one can check that the function $\mu_{U_{17}}-\mu_{3U_{17}}\in V(\cyclic{17})$ would not have this property if we replaced $U_n$ with $\langle 2 \rangle_{\cyclicstar{n}}$ in \cref{def:moves}.

\begin{proposition}\label{prop:v-obtained-by-moves}
    Let $G$ be an abelian group with no elements of order $2$, and let $A, B$ be finite multisets of $G$ such that $\mu_A-\mu_B\in V(G)$.
    Then there is a sequence of multisets $A=E_0, E_1, E_2, \dots, E_m=B$ such that each
    $E_i$ can be transformed into $E_{i+1}$ by applying one of the moves described in \cref{def:moves}.
\end{proposition}

\begin{proof}
Let $\mu\defeq \mu_A-\mu_B$, and decompose $\mu=\mu|_{G_T}+\mu|_{G \setminus G_T}$ as in \cref{lem:v-nontorsion-decomposition}.  For each $g \in G \setminus G_T$ such that $\mu(g)> 0$, apply the move $g \leadsto -g$ to $A$ exactly $\mu(g)$ times; call the resulting set $A'$.  Then $\mu_{A'}-\mu_B=\mu|_{G_T} \in V(G)$, so it suffices to prove the proposition under the assumption that $G$ is a torsion group; we will now suppose that this is the case.

For each $H\le G$ cyclic, let $A_H$ be the multiset of $H^\times$ with multiplicity function $\mu_{A_H}=\mu_A\cdot \mathds 1_{H^{\times}}$; define $B_H$ similarly.\footnote{Morally, $A_H$ is $A\cap H^\times$, but we avoid this notation as it is ambiguous in the context of multisets.}  Thanks to \cref{lem:v-decomposition}, we know that $\mu_{A_H}-\mu_{B_H}\in \tilde V(H)$, and it suffices to prove the lemma for each $A_H, B_H$ individually (since we can concatenate the resulting sequences of moves).
For $H\le G$ with $H\cong \cyclic{n}$ (for some $n\ge 3$ odd), the desired statement for $A_H, B_H$ is equivalent to the statement of the lemma for $G=\cyclic{n}$ under the assumption that $\mu:=\mu_A-\mu_B\in \tilde V(\cyclic{n})$. We will prove the latter statement.

Recall that $U_n^\pm=U_n\sqcup (-U_n)$. Take $g_1, g_2, \dots, g_\ell\in \cyclicstar{n}$ such that
    \begin{equation*}
        \cyclicstar{n} = 
        \bigsqcup_{i=1}^\ell g_i U_n^\pm.
    \end{equation*}
As observed in the proof of \cref{lem:dimension-tildev}, the quantity $\mu(g)+\mu(-g)$ is equal to some constant $t_i$ on each $g_i U_n^\pm$ and we have $\sum_i t_i=0$.  By applying moves $g\leadsto -g$ to $A$ and $B$, we can reduce to the situation where $\mu$ assumes the constant value $t_i$ on each $g_i U_n$ and vanishes on all of the $-g_iU_n$'s.
We now proceed by induction on the $\ell^1$-norm of $\mu$, which equals $|U_n|\sum_i |t_i|$.
If this $\ell^1$-norm is $0$, then  $A=B$ (since $\mu=0$) and there is nothing to prove.  If not, then there are indices $1\le \alpha, \beta\le \ell$ with $t_\alpha>0$ and $t_\beta<0$.  In particular, each element of $g_\alpha U_n$ appears with greater multiplicity in $A$ than in $B$, and each element of $g_\beta U_n$ appears with greater multiplicity in $B$ than in $A$. Obtain a new set $A'$ by applying the move $g_\alpha U_n \leadsto g_\beta U_n$ to $A$.  With $A$ replaced by $A'$, we still have the property that $\mu$ is constant on each $g_i U_n$ and vanishes on all of the $-g_iU_n$'s.  The values of $|t_\alpha|, |t_\beta|$ have been decreased by $1$, and the other $t_i$'s have remained constant, so the $\ell^1$-norm of $\mu$ has decreased and we can apply our induction hypothesis.
\end{proof}

\begin{corollary}\label{cor:v-generated-by-moves}
    Let $G$ be an abelian group with no elements of order $2$.
    Then the $\Z$-module $V(G)$ coincides with the $\Z$-submodule of $\Z^G$ generated by:
    \begin{itemize}
        \item $\mu_{\{g\}}-\mu_{\{-g\}}$, for $g \in G$;
        \item $\mu_{\iota(\alpha U_n)}-\mu_{\iota(\beta U_n)}$, for $n\ge 3$ odd, $\alpha, \beta\in\cyclicstar{n}$, and $\iota$ an embedding of $\cyclic{n}$ into $G$.
    \end{itemize}
\end{corollary}
\begin{proof}
Let $V'(G)$ denote the $\Z$-submodule generated by the functions described in the statement of the lemma.
The inclusion $V(G)\subseteq V'(G)$ is a consequence of \cref{prop:v-obtained-by-moves}. Indeed, given any $\mu\in V(G)$, let $A, B$ be the multisets of $G$ satisfying $\mu_A=\max\{0, \mu\}$ and $\mu_B=-\min\{0, \mu\}$; then $\mu=\mu_A-\mu_B$.  Now \cref{prop:v-obtained-by-moves} guarantees that there exists a sequence $A=E_0, E_1, \dots, E_m=B$ such that each $\mu_{E_{i+1}}-\mu_{E_i}$ coincides with one of the functions that generated $V'(G)$.  Hence
   \begin{equation*}
      \mu = \mu_A-\mu_B = \sum_{i=0}^{m-1} (\mu_{E_{i+1}}-\mu_{E_i}) \in V'(G),
   \end{equation*}
as desired.

For the inclusion $V'(G)\subseteq V(G)$, it suffices to show that the the functions mentioned in the statement belong to $V(G)$. It is immediate that $\mu_{\{g\}}-\mu_{\{-g\}}\in V(G)$, so we focus on showing that $\mu_{\iota(\alpha U_n)}-\mu_{\iota(\beta U_n)}\in V(G)$.  Since $\iota(\alpha U_n), \iota(\beta U_n)\subseteq \iota(\cyclicstar{n})$, \cref{lem:v-decomposition,lem:v-nontorsion-decomposition} tell us that it is enough to show that $\mu_{\alpha U_n}-\mu_{\beta U_n}\in \tilde V(\cyclic{n})$.  Let us verify separately the two conditions for membership in $\tilde V(\cyclic{n})$.  The condition that $\mu(g)+\mu(-g)=\mu(2g)+\mu(-2g)$ for all $g \in \cyclic{n}$ is equivalent to the condition that $\mu(g)+\mu(-g)$ is constant on each coset of $U_n$; this holds for $\mu_{\alpha U_n}-\mu_{\beta U_n}$ since it is the difference of two indicator functions of cosets of $U_n$.  The condition $\sum_{g\in\cyclicstar{n}} \mu(g)=0$ holds since $\sum_{g \in \cyclicstar{n}}(\mu_{\alpha U_n}(g)-\mu_{\beta U_n}(g))=|\alpha U_n| - |\beta U_n| = 0$.
\end{proof}

\section{Finite cyclic groups}\label{sec:finite-cyclic}
The purpose of this section is to prove the implication (i)$\implies$(ii) of \cref{thm:main} in the case where $G$ is a finite cyclic group. 

\begin{proposition}\label{prop:(i)-->(ii)-cyclic}
The implication (i)$\implies$(ii) of \cref{thm:main} holds when $G=\cyclic{n}$ for $n$ an odd natural number.
\end{proposition}

\subsection{Fourier-analytic reduction}
We begin by interpreting the map $\FS$ in terms of the Fourier transform.  In particular, we will see that $\FS$ ``acts like'' a linear map when viewed on the Fourier side.  The same perspective essentially appeared in \cite{CipriettiGlaudo2023}, just phrased in different language.  Let $G$ be a finite abelian group.  A \emph{character} of $G$ is a group homomorphism $\chi: G \to \C^\times$, and we let $\widehat{G}$ denote the group of all such characters.  If $f: G \to \C$ is a function, then its \emph{Fourier transform} $\hat f:\widehat G \to \C$ is the function
$$\hat f(\chi)\defeq \sum_{g \in G} f(g) \chi(-g).$$
See \cite[Part I]{Terras1999} for an overview of the Fourier transform for finite abelian groups.

Let $A$ be a finite multiset of $G$, and write $A=\{a_1, \ldots, a_q\}$ (with multiplicities).  
We have the identity
$$\mu_{\FS(A)}=(\mu_{\{0\}}+\mu_{\{a_1\}})*\cdots*(\mu_{\{0\}}+\mu_{\{a_q\}}),$$ 
where $*$ indicates the \emph{convolution} $(f_1*f_2)(g)\defeq \sum_{h \in G} f_1(h) f_2(g-h)$.  Notice that the term $\mu_{\{0\}}+\mu_{\{g\}}$ appears exactly $\mu_A(g)$ times on the right-hand side.  Taking Fourier transforms of both sides, we get
$$\hat \mu_{\FS(A)}(\chi)=\prod_{g \in G}(1+\chi(g))^{\mu_A(g)}.$$
For $A,B$ finite multisets of $G$ and $s \in G$, we have $\FS(A)=\FS(B)+s$ if and only if
\begin{equation*}
    \prod_{g\in G}(1+\chi(g))^{\mu_A(g)} = \hat \mu_{\FS(A)}(\chi)=\hat \mu_{\FS(B)}(\chi) \chi(s)
    =
    \chi(s)\prod_{g\in G}(1+\chi(g))^{\mu_B(g)}
\end{equation*}
for all characters $\chi\in \widehat G$. Since $G$ has no elements of order $2$, the expressions $(1+\chi(g))$ cannot vanish, so the previous equality is equivalent to
\begin{equation*}
    \prod_{g\in G}(1+\chi(g))^{\mu(g)}
    =
    \chi(s)
\end{equation*}
for all characters $\chi\in\widehat G$, where we have defined $\mu\defeq \mu_A-\mu_B$.

Now consider the case where $G=\cyclic{n}$ for some odd natural number $n$. 
Write $\omega_n\defeq  e^{2\pi i/n}$.  The characters of $\cyclic{n}$ are precisely the functions $\chi_j:g \mapsto \omega_n^{jg}$ for $j \in \cyclic{n}$.  For a finite multiset $S$ of $\cyclic{n}$, the following three pieces of data are equivalent: the multiset $S$; the function $\hat \mu_{S}$; and the collection of complex numbers $\hat\mu_{S}(\chi_d)$ for $d \mid n$. The equivalence of the first two amounts to the invertibility of the Fourier transform. The third piece of information is a restriction of the Fourier transform to a special set of characters. One can recover the value of the Fourier transform at an arbitrary character $\chi_j$ as follows. Let $d\defeq n/\gcd(j,n)$.  Then there is some $\sigma$ in the Galois group of $\Q[\omega_{d}]/\Q$ that sends $\omega_d=\omega_n^{n/d}$ to $\omega_n^j$. Since $\mu_S$ is integer-valued, we have $\hat\mu_S(\chi_j)=\sigma(\hat\mu_S(\chi_d))$ and thus we can recover all values of $\hat\mu_S$ from the values at the characters $\chi_d$ for $d\mid n$.

Thanks to what we observed in the previous two paragraphs, for $A,B$ finite multisets of $\cyclic{n}$ and $s \in \cyclic{n}$, we have $\FS(A)=\FS(B)+s$ if and only if
\begin{equation*}
    \prod_{j\in\cyclic{n}}(1+\omega_d^j)^{\mu(j)}
    =
    \omega_d^s
\end{equation*}
for all $d\mid n$.  With this characterization in mind, we make the following central definitions. 

\begin{definition}[{\cite[Definition 4.1]{CipriettiGlaudo2023}}]
    For $n\ge 1$ an odd natural number, let $K_n$ be the (multiplicative) subgroup of $\C^\times$ generated by $\{1 + \omega_n^j : 0 \le j < n\}$.
    Notice that we include $1 + \omega_n^0 = 2$ among the generators.
\end{definition}

\begin{definition}
Let $n\ge 1$ be an odd natural number, and define the group homomorphism $\widehat\FS:\Z^{\cyclic{n}} \cong \Z^n\to \oplus_{d\mid n} K_d$ by
    \begin{equation*}
        \widehat\FS(x)=\widehat\FS(x_0, x_1,\, \ldots, x_{n-1}) \defeq  \bigg(\prod_{j=0}^{n-1}(1+\omega_d^j)^{x_j}\bigg)_{d\mid n}.
    \end{equation*}
\end{definition}

The above discussion is summarized in the following lemma.

\begin{lemma}\label{lem:fourier}
    Let $n$ be an odd natural number.  Let $A, B$ be finite multisets of $\cyclic{n}$, and let $s \in \cyclic{n}$.  Then $\FS(A)=\FS(B)+s$ if and only if 
    \begin{equation*}
        \widehat\FS(\mu) = (\omega_d^{s})_{d\mid n},
    \end{equation*}
    where $\mu:=\mu_A-\mu_B$; notice that when this occurs, $n \mu$ lies in the kernel of $\widehat\FS$.
\end{lemma}

\subsection{Rank calculations}
In light of \cref{lem:fourier}, we wish to study the kernel of $\widehat{\FS}$.  Focusing on the kernel (rather than the preimages of all of the points $(\omega_d^{s})_{d\mid n}$) is advantageous because we will be able to apply rank-nullity dimension-counting arguments.  
The crux of our argument is the following important lemma from \cite{CipriettiGlaudo2023}, which tells us that the image of $\widehat\FS$ has full rank in its codomain.
\begin{lemma}[{\cite[Lemma 4.4]{CipriettiGlaudo2023}}] \label{lem:fshat-surjective}
    Let $n$ be an odd natural number. Then the image of the map $\widehat\FS:\Z^n\to \oplus_{d\mid n} K_d$ is a finite-index subgroup of $\oplus_{d\mid n} K_d$.
\end{lemma}

We prove a lower bound for the rank of the codomain of $\widehat\FS$.

\begin{lemma}\label{lem:rk-computation}
    We have $\rk(K_1)=1$.
    For $n\ge 3$ odd, we have
    \begin{equation*}
        \rk(K_n) \ge \frac{\varphi(n)}2 - \Big(\frac{\varphi(n)}{2|U_n|}-1\Big).
    \end{equation*}
\end{lemma}
The later results in this section imply that in fact equality holds, but at this point we can prove only the inequality.

\begin{proof}
If $n=1$, then $K_n=\langle 2\rangle_{\C^\times}\cong \Z$ has rank $1$.
    
Now, suppose $n\ge 3$. We have $(1+\omega_n)^{-1} = 1 + \omega_n^2 + \omega_n^4 + \cdots + \omega_n^{n-1}$, so $1+\omega_n$ is a unit of (the ring of integers of) $\Q(\omega_n)$. Analogously, $1+\omega_n^j$ is a unit for each $0<j<n$. By contrast, $1+\omega_n^0=2$ is not a unit. 
Since $2$ is the only non-unit among the generators of $K_n$, we have $K_n\cong \langle 2\rangle_{\C^\times} \oplus \tilde K_n$, where $\tilde K_n$ is the subgroup generated by $\{ 1+\omega_n^j:\, 1\le j < n\}$.  For each $j\in(\cyclic{n})^\times$, define the number 
    \begin{equation*}
        \xi_j \defeq  \prod_{\substack{d\mid n\\d\not=1}} \frac{1-\omega_d^j}{1-\omega_d} .
    \end{equation*}
Notice that there exists $j' \in \cyclicstar{n}$ such that $jj' \equiv 1 \pmod{n}$, and then
$$\xi_j^{-1}=\prod_{\substack{d\mid n\\d\not=1}} \frac{1-\omega_d}{1-\omega_d^j}=\prod_{\substack{d\mid n\\d\not=1}}\frac{1-\omega_d^{j j'}}{1-\omega_d^{j}}=\prod_{\substack{d\mid n\\d\not=1}}(1+\omega_d^j+\cdots+\omega_d^{j(j'-1)})$$
is an algebraic integer in $\Q(\omega_n)$, so $\xi_j$ is a unit. Also, the identities
    \begin{align*}
        \xi_{-j} = \xi_j\cdot\prod_{\substack{d\mid n\\d\not=1}} \omega_d^{-j} 
        = \xi_j\prod_{\substack{d\mid n\\d\not=1}} \frac{1+\omega_d^{-j}}{1+\omega_d^j} \quad \text{and} \quad 
        \xi_{2j}= \xi_j\cdot\prod_{\substack{d\mid n\\d\not=1}} (1+\omega_d^j)
    \end{align*}
show that $\xi_{2j}/\xi_j, \xi_{-j}/\xi_j\in \tilde K_n$ for all $j \in \cyclicstar{n}$.

Recall that $U_n^\pm=U_n\sqcup (-U_n)$. Let $\ell =  \frac{\varphi(n)}{2|U_n|}$, and take $g_1, g_2,\dots, g_{\ell}\in(\cyclic{n})^\times$, with $g_\ell=1,$ such that
    \begin{equation*}
        (\cyclic{n})^{\times} = \bigsqcup_{i=1}^{\ell} g_iU_n^\pm.
    \end{equation*}
The observations from the previous paragraph show that $\langle \tilde K_n, \xi_{g_1}, \xi_{g_2}, \dots, \xi_{g_{\ell-1}}\rangle_{\C^\times}$ contains $\{\xi_j:\ j\in(\cyclic{n})^\times\}$.  Thanks to \cite[Theorem 8.3 and Theorem 4.12]{washington97}, we know that $\{\xi_j:\ j\in(\cyclic{n})^\times\}$ generates a finite-index subgroup of the units of $\Q(\omega_n)$, and we conclude that $\langle \tilde K_n, \xi_{g_1}, \xi_{g_2}, \dots, \xi_{g_{\ell-1}}\rangle_{\C^\times}$ is also a finite-index subgroup.  The rank of the group of units is $\varphi(n)/2-1$ (as a consequence of Dirichlet's Unit Theorem \cite[Theorem 38]{marcus77}), so $\rk(\tilde K_n)\ge \varphi(n)/2-1 - (\ell-1)$. Thus
    \begin{equation*}
        \rk(K_n)=1+\rk(\tilde K_n)\ge 1+\frac{\varphi(n)}2-\ell=\frac{\varphi(n)}2 - \Big(\frac{\varphi(n)}{2|U_n|}-1\Big),
    \end{equation*}
as desired.
\end{proof}

\subsection{Proof of (i)\texorpdfstring{$\implies$}{ implies }(ii) for finite cyclic groups}
The last piece of the puzzle is checking how the moves of \cref{def:moves} interact with $\FS$.

\begin{lemma}\label{lem:fs-of-moves}
    Let $G$ be an abelian group with no elements of order $2$.
    \begin{enumerate}
        \item For any $g\in G$, we have
        \begin{equation*}
            \FS(\{g\}) = \FS(\{-g\}) + g.
        \end{equation*}
        \item For any odd $n\ge 3$, any $\alpha,\beta\in \cyclicstar{n}$, and any embedding $\iota:\cyclic{n}\to G$, we have
        \begin{equation*}
            \FS\big(\iota(\alpha U_n)\big) = \FS\big(\iota(\beta U_n)\big) + s
        \end{equation*}
        for some $s\in G$.
    \end{enumerate}
In particular, if $A$ is a finite multiset of $G$ and $A'$ is obtained from $A$ by applying one of the moves described in \cref{def:moves}, then $\FS(A')=\FS(A)+s$ for some $s \in G$.
\end{lemma}
\begin{proof}
Part (1) is immediate, so we will prove only part (2).  Let $k \in \N$ be minimal such that $2^k \in \{1,-1\}$ modulo $n$, so that $U_n=\{1,2,4, \ldots, 2^{k-1}\}$.  Then $$\FS(U_n)=\{0,1, \ldots, 2^k-1\},$$
and it remains to characterize the image of this set modulo $n$.  If $2^k=1\pmod{n}$, then all nonzero elements of $\cyclic{n}$ get the same multiplicity and $0$ gets multiplicity $1$ greater. 
Since $\FS(\alpha U_n)=\alpha \FS(U_n)$, the same holds if $U_n$ is scaled by an element of $\cyclicstar{n}$, and hence $\FS(\alpha U_n)=\FS(\beta U_n)$ for all choices of $\alpha, \beta\in\cyclicstar{n}$. Then, since $\FS(\iota(\alpha U_n)) = \iota(\FS(\alpha U_n))$, the desired statement holds with $s=0$.

If $2^k=-1\pmod{n}$, then all elements of $\cyclic{n}$ get the same multiplicity  except for $-1$, which has multiplicity $1$ smaller. In $\FS(\alpha U_n)$, the element with smaller multiplicity is $-\alpha$ rather than $-1$, so $\FS(\alpha U_n)=\FS(\beta U_n)+(\beta-\alpha)$, and the desired statement holds with $s=\iota(\beta-\alpha)$.

The final statement of the lemma follows from the observation that $\FS(B \cup C)=FS(B)+\FS(C)$ (where $+$ here denotes the Minkowski sum with multiplicity).
\end{proof}

We are finally ready to prove \cref{prop:(i)-->(ii)-cyclic}. 

\begin{proof}[Proof of \cref{prop:(i)-->(ii)-cyclic}]
By \cref{lem:fshat-surjective}, we know that the image of $\widehat\FS$ has finite index in its codomain.  Hence, \cref{lem:rk-computation} tells us that
    \begin{align*}
        \rk(\ker\widehat\FS) &= \rk(\Z^n) - \rk\Big(\bigoplus_{d\mid n}K_d\Big) = n - \sum_{d\mid n} \rk(K_d) 
        \\
        &= n - 1 - \sum_{\substack{d\mid n\\d\not=1}} \frac{\varphi(d)}2 + \sum_{\substack{d\mid n\\d\not=1}} \Big(\frac{\varphi(d)}{2|U_d|}-1\Big)
        = \frac{n-1}2 + \sum_{\substack{d\mid n\\d\not=1}} \Big(\frac{\varphi(d)}{2|U_d|}-1\Big)
        .
    \end{align*}
\cref{lem:fs-of-moves,lem:fourier} tell us that $n\cdot V(\cyclic{n})$ is a $\Z$-submodule of $\ker\widehat\FS$, and \cref{cor:dimension-v} tells us that $\rk(\ker\widehat\FS) = \rk(V(\cyclic{n})) = \rk(n\cdot V(\cyclic{n}))$. Since all modules in question are free (as they are submodules of $\Z^n$), this implies that there is a natural number $n'\ge 1$ such that $n'\cdot \ker\widehat\FS$ is a submodule of $V(\cyclic{n})$.

Finally, let $A,B$ be finite multisets of $\cyclic{n}$ with $\FS(A)=\FS(B)+s$ for some $s \in \cyclic{n}$, and let $\mu:=\mu_A-\mu_B$.  \cref{lem:fourier} tells us that $n \mu \in \ker \widehat\FS$, so $n'n\mu \in V(\cyclic{n})$.  Since $\mu$ is $\Z$-valued, we conclude that $\mu\in V(\cyclic{n})$.
\end{proof}

\section{From finite cyclic groups to general groups}\label{sec:radon}
We will now prove the implication (i)$\implies$(ii) of \cref{thm:main} for general abelian groups $G$.  We will prove this implication in several stages, for progressively larger classes of abelian groups.

\subsection{\texorpdfstring{From $\cyclic{n}$ to $(\cyclic{n})^r$}{From Z/nZ to powers of Z/nZ}}

The most technical stage of our build-up is deducing information about $(\cyclic{n})^r$ from information about $\cyclic{n}$.

\begin{proposition}\label{prop:(i)-->(ii)-radon}
The implication (i)$\implies$(ii) of \cref{thm:main} holds when $G=(\cyclic{n})^r$ for $n\ge 1$ an odd natural number and $r\ge 1$ a natural number.
\end{proposition}

Our main tool is the discrete Radon transform as introduced and studied in \cite[Section 5]{CipriettiGlaudo2023}.

\begin{definition}[{\cite[Definition 5.1]{CipriettiGlaudo2023}}]\label{def:radon}
    Let $n, r \ge 1$ be natural numbers. The \emph{Radon transform} of a function $f:(\cyclic{n})^r\to\C$, is the function $Rf = R_{n, r}f:\Hom((\cyclic{n})^r,\, \cyclic{n})\times \cyclic{n}\to\C$ given by
    \begin{equation*}
        Rf(\psi, c) \defeq (\psi_*f)(c)=\sum_{\substack{x\in (\cyclic{n})^r\\ \psi(x) = c}} f(x),
    \end{equation*} 
    for all homomorphisms $\psi:(\cyclic{n})^r\to\cyclic{n}$ and all $c\in \cyclic{n}$.
\end{definition}

This operator is called the Radon transform in analogy with the classical Radon transform in $\R^r$ that computes the integral of a given function over affine subspaces.  In $\R^r$, affine subspaces are fibers of linear operators from $\R^r$ to $\R$; we take the latter characterization as motivation for our Radon transform on $(\cyclic{n})^r$.  Indeed, $Rf(\psi, c)$ is the sum (\emph{integral}) of the values of $f$ over a fiber (\emph{hyperplane}) of a homomorphism (\emph{linear operator}) from $(\cyclic{n})^r$ to $\cyclic{n}$. See \cite[Section 5]{CipriettiGlaudo2023} for a more in depth comparison between the discrete Radon transform and the classical Radon transform.

We will make use of the following inversion formula for the Radon transform.

\begin{theorem}[{\cite[Theorem 1.2]{CipriettiGlaudo2023}}]\label{thm:radon-inversion}
    Let $n, r\ge 1$ be natural numbers, and let $f:(\cyclic{n})^r\to\C$ be a function.  Then
    \begin{equation*}
        f(x)
        =
        \frac{1}{n^{r-1}\varphi(n)}
        \sum_{\psi\in \Hom((\cyclic{n})^r, \cyclic{n})}
        Rf(\psi, \psi(x))\prod_{p\mid \psi}
        (1-p^{r-1})
    \end{equation*}
for all $x \in (\cyclic{n})^r$.  Here the product runs over all prime divisors $p$ of $n$ such that $p$ divides all of the elements in the image of $\psi$ (or, equivalently, such that $\psi$ takes values in $p\Z/n\Z$);  $\varphi$ denotes Euler's totient function.
\end{theorem}

The proof of \cref{prop:(i)-->(ii)-radon} combines \cref{prop:(i)-->(ii)-cyclic} (the implication (i)$\implies$(ii) for $G$ finite cyclic), the inversion formula for the Radon transform, and the properties of pushforwards and pullbacks shown in \cref{lem:v-functoriality}.

\begin{proof}[Proof of \cref{prop:(i)-->(ii)-radon}]
For $S$ a finite multiset of $(\cyclic{n})^r$, $\psi \in \Hom((\cyclic{n})^r, \cyclic{n})$, and $c \in \cyclic{n}$, we have
$$R\mu_S(\psi, c)=\psi_*\mu_S(c)=\mu_{\psi(S)}(c).$$
Suppose $A,B$ are finite multisets of $(\cyclic{n})^r$ such that $\FS(A)=\FS(B)+s$ for some $s \in (\cyclic{n})^r$.  Write $\mu\defeq \mu_A-\mu_B$.  Then, for each $\psi \in \Hom((\cyclic{n})^r, \cyclic{n})$, we have $$\FS(\psi(A))=\psi(\FS(A))=\psi(\FS(B)+s)=\psi(\FS(B))+\psi(s)=\FS(\psi(B))+\psi(s).$$
\cref{prop:(i)-->(ii)-cyclic} tells us that $$\psi_*\mu=R\mu(\psi, \emptyparam)=\mu_{\psi(A)}-\mu_{\psi(B)} \in V(\cyclic{n}),$$
while \cref{thm:radon-inversion} tells us that
$$\mu = \sum_{\psi\in \Hom((\cyclic{n})^r, \cyclic{n})} \lambda_\psi \psi^*\psi_*\mu $$
for some rational constants $\lambda_\psi$. Applying \cref{lem:v-functoriality} and recalling that $\mu$ is $\Z$-valued, we deduce that $\mu\in V((\cyclic{n})^r)$.
\end{proof}

\subsection{Finite abelian groups}
The following lemma, when combined with \cref{prop:(i)-->(ii)-radon}, establishes the implication (i)$\implies$(ii) of \cref{thm:main} for all finite abelian groups $G$ with no $2$-torsion.

\begin{lemma}\label{lem:subgroup}
Let $G$ be an abelian group with no elements of order $2$, and let $G'$ be a subgroup of $G$.  If the implication (i)$\implies$(ii) of \cref{thm:main} holds for $G$, then it also holds for $G'$.
\end{lemma}

\begin{proof}
Let $A, B$ be two finite multisets of $G'$ such that $\FS(A)=\FS(B)+s$. Then, by interpreting $A, B$ as multisets of $G$, we find that $\mu_A-\mu_B\in V(G)$.  By applying \cref{lem:v-functoriality} to the embedding of $G'$ into $G$, we deduce that $\mu_A-\mu_B\in V(G')$ as desired.
\end{proof}

\begin{proposition}\label{prop:(i)-->(ii)-finite}
The implication (i)$\implies$(ii) of \cref{thm:main} holds when $G$ is a finite abelian group with no elements of order $2$.
\end{proposition}

\begin{proof}
The Fundamental Theorem of Finitely Generated Abelian groups \cite[Chapter I, Section 8]{lang2002} tells us that there exist odd natural numbers $n_1, \ldots, n_r$ with each $n_i$ dividing $n_{i+1}$ such that $G\cong\cyclic{n_1} \oplus \cdots \oplus \cyclic{n_r}$. Then $G$ is a subgroup of $(\cyclic{n_r})^r$.  \cref{prop:(i)-->(ii)-radon} ensures that the implication (i)$\implies$(ii) of \cref{thm:main} holds for $(\cyclic{n_r})^r$, and then \cref{lem:subgroup} tells us that the implication (i)$\implies$(ii) of \cref{thm:main} also holds for $G$.
\end{proof}

\subsection{Finitely generated abelian groups}

We now extend the results of the previous subsection to finitely generated abelian groups.

\begin{proposition}\label{prop:(i)-->(ii)-finitely-generated}
The implication (i)$\implies$(ii) of \cref{thm:main} holds for every finitely generated abelian group with no elements of order $2$.
\end{proposition}

We will show that if the implication (i)$\implies$(ii) holds for an abelian group $G$, then it also holds for $G \times \Z$; by the Fundamental Theorem of Finitely Generated Abelian Groups, this will imply that (i)$\implies$(ii) holds for all finitely generated abelian groups.  We will require the following lemma.

\begin{lemma}[{\cite[Lemma 6.1]{CipriettiGlaudo2023}}]\label{lem:remove-common-summand}
Let $G$ be an abelian group with no elements of order $2$.  If $A, A', B$ are finite multisets of $G$ such that $A+\FS(B)=A'+\FS(B)$, then $A=A'$.
\end{lemma}
We remark that in this lemma it is essential to include the hypothesis that $G$ has no $2$-torsion.

\begin{corollary}\label{cor:remove-common-summand-shift}
    Let $G$ be an abelian group with no elements of order $2$.
    Let $A'\subseteq A$ and $B'\subseteq B$ be finite multisets of $G$ such that $\FS(A)$ is a shift of $\FS(B)$ and $\FS(A')$ is a shift of $\FS(B')$. Then $\FS(A\setminus A')$ is a shift of $\FS(B\setminus B')$. 
\end{corollary}
\begin{proof}
    Take $s, s'\in G$ such that $\FS(A)=\FS(B)+s$ and $\FS(A')=\FS(B')+s'$. Then
    \begin{align*}
        \FS(A\setminus A') + \FS(A') &= \FS(A) = \FS(B) + s
        \\
        &= \FS(B\setminus B') + \FS(B') + s = \FS(B\setminus B') + (s-s') + \FS(A'),
    \end{align*}
    and, applying \cref{lem:remove-common-summand}, we deduce that $\FS(A\setminus A') = \FS(B\setminus B') + (s-s')$.
\end{proof}

\begin{proposition}\label{prop:(i)-->(ii)-times-Z}
Let $G$ be an abelian group with no elements of order $2$.  If the implication (i)$\implies$(ii) of \cref{thm:main} holds for $G$, then it also holds for $G\oplus\Z$.
\end{proposition}

\begin{proof}
For $S$ a finite multiset of $G \oplus \Z$ and $z \in \Z$, we define the ``slice'' of $S$ at height $z$ to be the set
$$S_z\defeq S \cap (G \times \{z\}).$$
Let $A,B$ be finite multisets of $G \oplus \Z$ such that $\FS(A)=\FS(B)+(s,y)$ for some $s \in G$ and $y \in \Z$.  We will show by induction on the size of $A,B$ that $\mu_A-\mu_B\in V(G\oplus\Z)$.  Let $w$ be the largest integer such that $\FS(A)_w$ is nonempty (equivalently, such that $\FS(B)_{w-y}$ is nonempty).  Then $\FS(A)_w=(g,w)+\FS(A_0)$ for some $g \in G$, and $\FS(B)_{w-y}=(h,w-y)+\FS(B_0)$ for some $h \in G$.  Hence
$$(g,w)+\FS(A_0)=\FS(A)_w=\FS(B)_{w-y}+(s,y)=(h,w-y)+\FS(B_0)+(s,y),$$
i.e.,
$$\FS(A_0)=\FS(B_0)+(s+h-g,0).$$
Viewing $A_0, B_0$ as multisets of $G$ and using our assumption on $G$, we conclude (applying \cref{lem:v-functoriality} to the embedding of $G$ into $G \oplus \Z$) that $\mu_{A_0}-\mu_{B_0} \in V(G \oplus \Z)$.
\cref{cor:remove-common-summand-shift} tells us that $\FS(A \setminus A_0)$ is still a shift of $\FS(B \setminus B_0)$.
If $A_0, B_0$ are nonempty, then our induction hypothesis tells us that $\mu_{A \setminus A_0}-\mu_{B \setminus B_0} \in V(G \oplus \Z)$, and we conclude that $\mu_A-\mu_B \in V(G \oplus \Z)$, as desired.  So it remains to consider the case where $A_0, B_0$ are empty.

In this case, $A_w=\{(g,w)\}$ consists of only a single element.  Let $w'<w$ denote the second-largest integer such that $\FS(A)_{w'}$ is nonempty; such a $w'$ exists since otherwise we would have $A \subseteq G \times \{0\}$, contrary to our assumption that $A_0$ is empty.  Pick any $(g',w')$ in $A_{w'}$.  The key observation is that $A$ must contain either $p\defeq (g-g',w-w')$ or $-p=-(g-g',w-w')$.  Indeed, $(g,w)$ is the sum of all elements in $A$ with positive second coordinate, and $(g', w')$ is obtained from $(g,w)$ either by removing from this sum an element with positive second coordinate (namely, $p$) or by adding to this sum an element of $A$ with negative second coordinate (namely, $-p$).
Since $\FS(B)$ is a translate of $\FS(A)$, the same argument shows that $B$ also contains either $p$ or $-p$.  
Choose $a,b\in \{p, -p\}$ such that $a\in A$ and $b\in B$. Notice that either $a=b$ or $a=-b$. 
Either way, we have (see \cref{cor:v-generated-by-moves}) 
$$\mu_{\{a\}}-\mu_{\{b\}} \in V(G\oplus \Z).$$  
Moreover, \cref{cor:remove-common-summand-shift} tells us that $\FS(A \setminus \{a\})$ is still a shift of $\FS(B \setminus \{b\})$, so 
$$\mu_{A \setminus \{a\}}-\mu_{B \setminus \{b\}} \in V(G \oplus \Z)$$ 
by induction.  Summing the previous two expressions gives $\mu_A-\mu_B \in V(G \oplus \Z)$.
\end{proof}

The proof of \cref{prop:(i)-->(ii)-finitely-generated} is a simple combination of \cref{prop:(i)-->(ii)-finite,prop:(i)-->(ii)-times-Z}.
\begin{proof}[Proof of \cref{prop:(i)-->(ii)-finitely-generated}]
    The Fundamental Theorem of Finitely Generated Abelian Groups \cite[Chapter I, Section 8]{lang2002} tells us that a finitely generated abelian group $G$ without elements of order $2$ can be written as $G\cong G' \oplus \Z^r$ where $G'$ is a finite group (with no elements of order $2$) and $r$ is a nonnegative integer. 
    \cref{prop:(i)-->(ii)-finite} tells us that the implication (i)$\implies$(ii) of \cref{thm:main} holds for $G'$, then applying \cref{prop:(i)-->(ii)-times-Z} $r$ times we obtain that it holds also for $G$.
\end{proof}

\begin{remark}
Let us sketch an alternative proof of \cref{prop:(i)-->(ii)-finitely-generated} which is based on a standard \emph{truncation} argument.  We can write our finitely generated abelian group as $G=G'\times \Z^r$ where $G'$ is finite. Let $A, B$ be finite multisets of $G$ such that $\FS(A)=\FS(B)+s$ for some $s\in G$. Choose some large $N\in\N$ such that $A, B\subseteq G'\times [-N, N]^r$, and define $M=4N+1$. Consider the projection $\pi:G\cong G'\oplus\Z^r\to G'\oplus (\cyclic{M})^r$ which is the identity on the first factor and the canonical projection on the second factor.  Notice that $\pi$ is injective on the support of $\mu_A-\mu_B$. Thanks to \cref{prop:(i)-->(ii)-finite}, we know that $\pi_*(\mu_A-\mu_B)\in V(G'\oplus (\cyclic{M})^r)$. This in turn implies (by direct computation) that $\mu_A-\mu_B \in V(G'\oplus \Z^r)$.
\end{remark}

\subsection{General abelian groups}\label{sec:general-ab}
Finally, we treat the case of general abelian groups.

\begin{proposition}\label{prop:prop:(i)-->(ii)-all}
The implication (i)$\implies$(ii) of \cref{thm:main} holds for every abelian group with no elements of order $2$.
\end{proposition}

\begin{proof}
Let $G$ be an abelian group with no elements of order $2$.  Let $A,B$ be finite multisets of $G$ such that $\FS(A)=\FS(B)+s$ for some $s \in G$.  Let $G'$ be the subgroup generated by the elements of $A$ and $B$; clearly $G'$ is finitely generated, and we can consider $A,B$ as multisets of $G'$. \cref{prop:(i)-->(ii)-finitely-generated} tells us that $\mu_A-\mu_B\in V(G')$, and then \cref{lem:v-functoriality} (applied to the embedding of $G'$ into $G$) implies that $\mu_A-\mu_B\in V(G)$.
\end{proof}

\section{Final deduction of the main theorems}\label{sec:final}
We now complete the circle for the proof of \cref{thm:main}.

\begin{proof}[Proof of \cref{thm:main}]
The implication (i)$\implies$(ii) is \cref{prop:prop:(i)-->(ii)-all}; the implication (ii)$\implies$(iii) follows from \cref{prop:v-obtained-by-moves,lem:fs-of-moves}; and the implication (iii)$\implies$(i) is immediate.
\end{proof}

We also deduce \cref{thm:main-no-shift}, the no-shift version of \cref{thm:main}.

\begin{proof}[Proof of \cref{thm:main-no-shift}]
Let $A,B$ be finite multisets of $G$.  Each of the conditions in \cref{thm:main-no-shift} implies the corresponding condition in \cref{thm:main}, so it suffices to consider the case where $\FS(A)=\FS(B)+s$ for some $s$ in $G$; suppose this is the case. 

Condition (i) asserts that $s=0$. To handle condition (iii), notice that $s=s_0+\cdots+s_{m-1}$.  To handle condition (ii), it remains to show that $s=0$ if and only if $\sum_{g\in G} \mu(g)g=0$, i.e., if and only if $\sum A=\sum B$.  Let $q\defeq |A|=|B|$.  Then
$$2^{q-1}\sum A=\sum \FS(A)=\sum \FS(B)+2^qs=2^{q-1}\sum B+2^qs,$$
where the first equality holds because each element of $A$ belongs to exactly half of the $2^q$ subset sums of $A$ (and likewise for the last equality).  So $s=0$ if and only if $\sum A=\sum B$.
\end{proof}



\section*{Acknowledgments} 
We thank the anonymous referee for their comments on an earlier version of this paper.

\bibliographystyle{amsplain}


\begin{dajauthors}
\begin{authorinfo}[fede]
  Federico Glaudo\\
  School of Mathematics, Institute for Advanced Study\\
  Princeton, NJ, USA\\
  fglaudo \imageat{}ias \imagedot{}edu
\end{authorinfo}
\begin{authorinfo}[noah]
  Noah Kravitz\\
  Department of Mathematics, Princeton University\\
  Princeton, NJ, USA\\
  nkravitz\imageat{}princeton\imagedot{}edu
\end{authorinfo}
\end{dajauthors}

\end{document}